\documentclass[11pt]{amsart}
\usepackage{amsthm}
\usepackage{filecontents}
\usepackage{graphicx}
\usepackage{xcolor}
\usepackage{bibentry}
\nobibliography*\usepackage{amssymb}
\usepackage{amsmath}
\usepackage{amsfonts} 
\usepackage[all]{xy}
\usepackage[numbers, square, sort&compress]{natbib}
\newcommand{\excise}[1]{}

\newcommand \rad {\operatorname{rad}}%
\newcommand \ke {\operatorname{Ker}}
\newcommand \im {\operatorname{Im}}
\newcommand \M {\mathcal{M}}
\newcommand \R {\mathcal{R}}

\newcommand \+ {\texttt{+}}
\newcommand \p {\mathcal{\large{\texttt{.}}}}

\newtheorem{thm}{Theorem}
\newtheorem{lemma}[thm]{Lemma}

\newtheorem{cor}[thm]{Corollary}

\theoremstyle{definition}
\newtheorem{example}[thm]{Example}

\newtheorem{conv}[thm]{Convention}

        {\begin{list}
                {\noindent\makebox[0mm][r]{\arabic{enumi}}}
                {\leftmargin=5.5ex \usecounter{enumi}}
        }
        {\end{list}}

        {\begin{list}
                {\noindent\makebox[0mm][r]{(\roman{enumi})}}
                {\leftmargin=5.5ex \usecounter{enumi}}
        }
        {\end{list}}

\newcommand{\Hom}{\operatorname{Hom}}

\def\<{\langle}
\def\>{\rangle}

\def\cocoa
  {{\hbox{\rm C\kern-.13em o\kern-.07em C\kern-.13em o\kern-.15em A}}}

\begin{document}
\mbox{}
\vspace{-4.65ex}
\title[A Homology Theory for the Semimodules of Radical Submodules]{A Homology Theory for the Semimodules of Radical Submodules}
\author[M. Safaeipour]{Mahboubeh Safaeipour}
\address[Mahboubeh Safaeipour]{}
\curraddr{Department of Mathematics\\University of Birjand\\Birjand, Iran.}
\email{mahboubehsafaeipour@birjand.ac.ir}
\author[H. F. Moghimi]{Hosein Fazaeli Moghimi}
\address[Hosein Fazaeli Moghimi]{}
\curraddr{Department of Mathematics\\University of Birjand\\Birjand, Iran.}
\email{hfazaeli@birjand.ac.ir}
\author[F. Rashedi]{Fatemeh Rashedi}
\address[Fatemeh Rashedi]{}
\curraddr{Department of Mathematics\\National University of Skills (NUS)\\Tehran, Iran.}
\email{frashedi@nus.ac.ir}

\subjclass[2010]{18G99, 13C99, 16Y60}
\keywords{Semimodule of radical submodules, Map of complexes of modules, Radical homology, Homotopy, Radical projective semimodule}

\maketitle
\begin{abstract}
Let $R$ be a commutative ring with identity, and let $\R(R)$ denote the semiring of radical ideals of $R$. The radical functor $\R$, from the category of $R$-modules $R{-}\boldsymbol{\sf{Mod}}$ to the category of $\R(R)$-semimodules $\R(R){-}\boldsymbol{\sf{Semod}}$, maps any complex $\M=(M_n, f_n)_{n\geq 0}$ of $R$-modules to a complex $\R(\M)=(\R(M_n), \R(f_n))_{n\geq 0}$ of $\R(R)$-semimodules, where  $\R(M_n)$ consists of radical submodules of $M_n$, and the $\R(R)$-semimodule homomorphisms $\R(f_n):\R(M_n)\rightarrow \R(M_{n-1})$ are defined by $\R(f_n)(N)=\rad(f_n(N))$. The $n$-th radical homology of the complex $(\R(M_n), \R(f_n))_{n\geq 0}$, denoted $H_n(\R(\M))$, consists of radical submodules $N$ of $M_n$ such that $f_n(N)$ is contained in the radical of the zero submodule of $M_{n-1}$, and two such radical submodules are equivalent under the Bourne relation modulo the image of $\R(f_{n+1})$. $H_n(\R(-))$  is regarded as a covariant functor from the category $\boldsymbol{\sf{Ch}}(R{-}\boldsymbol{\sf{Mod}})$ of chain complexes of $R$-modules to $\R(R){-}\boldsymbol{\sf{Semod}}$, which acts identically on any pair of homotopic maps of complexes of $R$-modules. In particular, if $\M$ and $\M'$ are homotopically equivalent, then $H_n(\R(\M))$ and $H_n(\R(\M'))$ are isomorphic $\R(R)$-semimodules. We provide conditions under which $H_n(\R(-))$ induces a long exact sequence of radical homology modules for any short exact sequence of complexes of $R$-modules, and satisfies the naturality condition for exact homology sequences. Finally, we introduce a projective resolution for an $R$-module $M$ based on $\R(R)$-semimodules and give conditions under which such a projective resolution exists and is unique up to a homotopy. 

\end{abstract} 
\section{Introduction and preliminaries}
\indent 
In this paper, all semirings are assumed to be commutative. A commutative semiring $\mathsf{R}$ is defined as a non-empty set equipped with two binary operations,  $``+"$ and $``\cdot"$, where $(\mathsf{R},+,0_\mathsf{R})$ and $(\mathsf{R},\cdot, 1_R)$ form commutative monoids such that 
\begin{center}
$r\cdot(r_1+r_2)=r\cdot r_1+r\cdot r_2$  and  $r\cdot 0_R=0_R$ 
\end{center}
for all $r, r_1, r_2\in \mathsf{R}$. A commutative monoid $(\mathsf{M},+,0_\mathsf{M})$ is an $\mathsf{R}$-\textit{semimodule} if $\mathsf{R}$ acts on $\mathsf{M}$ like a ring on a unital module, and furthermore, for any
 $r\in\mathsf{R}$ and $m\in\mathsf{M}$, it is assumed that $r.m=m.r$ and $r.0_\mathsf{M}=0_\mathsf{R}.m=0_\mathsf{M}$. If no ambiguity arises, we omit the subscripts of $0_\mathsf{R}$, $1_\mathsf{R}$ and $0_\mathsf{M}$ and simply denote them by $0$ and $1$. A non-empty subset $\mathsf{N}$ of an  $\mathsf{R}$-semimodule  $\mathsf{M}$ is an $\mathsf{R}$-subsemimodule of $\mathsf{M}$, if it remains an $\mathsf{R}$-semimodule under the  $``+"$ and $``\cdot"$. For any subsemimodule $\mathsf{N}$ of an $\mathsf{R}$-semimodule $\mathsf{M}$, we mean by the Bourne relation $\mathsf{M}$ modulo $\mathsf{N}$ that $m\equiv_\mathsf{N} m'$ if and only if  $m+n=m'+n'$ for some $n,n'\in N$. The equivalence class of any $m\in M$ with respect to this relation is denoted by $m/\mathsf{N}$, and the collection of these equivalence classes is denoted by $\mathsf{M}/\mathsf{N}$. The addition and scalar multiplication of $\mathsf{R}$ on $\mathsf{M}/\mathsf{N}$ is defined by $m/\mathsf{N}+m'/\mathsf{N}=(m+m')/\mathsf{N}$ and $r\cdot(m/\mathsf{N})=(r\cdot m)/\mathsf{N}$ for all $m,m'\in\mathsf{M}$ and $r\in\mathsf{R}$.\\ \indent
Let $M$ be an $R$-module. A submodule $N$ of $M$ is called a radical submodule if $\operatorname{rad}(N)=N$, where $\operatorname{rad}(N)$, the radical of $N$, is the intersection of all prime submodules of $M$ containing $N$. If no prime submodule contains $N$, we set  $\operatorname{rad}(N)=M$. Also, $M$ is said to be a radical module if its zero submodule is a radical submodule (for more details on radical submodules, see \cite{Alkan,ms} for example).  In the case of ideals, we denote the radical of an ideal $I$ of $R$ by $\sqrt{I}$. Note that the set of radical ideals of $R$, denoted $\R(R)$, is a semiring with the operations $I{\texttt{+}} J=\sqrt{I+J}$ and $I\p J=\sqrt{IJ}$, where $I+J$ and $IJ$ are the usual addition and multiplication of ideals $I$ and $J$, respectively.  Moreover, the set of radical submodules of $M$, denoted $\R(M)$, is an $\R(R)$-semimodule with the operations $N{\texttt{+}} L=\operatorname{rad}(N+L)$ and $I\p N=\operatorname{rad}(IN)$ for $N, L\in \R(M)$ and $I\in\R(R)$, where $N+L$ denotes the usual addition of submodules $N$ and $L$ and $IN$ consists of all finite summs $\sum{a_{\lambda}n_{\lambda}}$ for $a_{\lambda}\in I$ and $n_{\lambda}\in N$) \cite[Theorem 2]{smr}. \\ \indent
Roughly, the concept of a semiring (or semimodule) homomorphism parallels the definition of ring (or module) homomorphisms. For two $R$-semimodules $\mathsf{M}$ and $\mathsf{M'}$, the notation $\operatorname{Hom}(\mathsf{M},\mathsf{M'})$ represents the collection of  $\mathsf{R}$-semimodule homomorphisms from $\mathsf{M}$ to $\mathsf{M'}$, which naturally forms an $\mathsf{R}$-semimodule. In \cite[Theorem 4]{smr}, we showed that, $\R$ is a covariant  functor from $R{-}\boldsymbol{\sf{Mod}}$  (the category of $R$-modules)  to $\R(R){-}\boldsymbol{\sf{Semod}}$ (the category of $\R(R)$-semimodules) which maps any $R$-module homomorphism $f:M\rightarrow M'$ to the $\R(R)$-semimodule homomorphism $\R(f):\R(M)\rightarrow \R(M')$ defined as $\R(f)(N)=\rad(f(N)$ for any radical submodule $N$ of $M$. In particular, by considering an exactness for semimodule sequences similar to that of modules over commutative rings, $\R$ is an exact functor \cite[Lemma 16]{smr}. We further showed that for any two $R$-modules $P$ and $E$, $\Hom_{\R(R)}(R(P),R(-))$ and $\Hom_{\R(R)}(R(-),R(E))$ are left exact on any short exact sequence $0\rightarrow M' \xrightarrow{f} M\xrightarrow{g} M''\rightarrow 0$, where $M''$ is radical. Subsequently, in \cite[Theorem 8]{smr1}, we investigated the conditions under which the natural tensor functor $\R(-)\otimes_{\R(R)} \R(T)$ (for an $R$-module $T$) preserves module exact sequences, where the approach to tensor products for semimodules over commutative semirings parallels that for modules over commutative rings. In this article, we extend this investigation toward developing a homology theory for semimodules of radical submodules. A  \textit{complex} $\mathcal{M}=(\mathsf{M}_n, f_n)_{n\geq0}$ of $\mathsf{R}$-semimodules is a sequence $\{\mathsf{M}_n\}_{n\geq0}$ of $\mathsf{R}$-semimodules and a sequence $\{f_n\}_{n\geq0}$ of $\mathsf{R}$-semimodule homomorphisms $f_n:\mathsf{M}_n\rightarrow \mathsf{M}_{n-1}$ such that $f_n f_{n+1}=0$ $(n\geq 0)$. We denote such a complex  by the symbol $$\mathcal{M}:\cdots\rightarrow \mathsf{M}_n\xrightarrow{f_n}\mathsf{M}_{n-1}\xrightarrow{f_{n-1}} \mathsf{M}_{n-2}\xrightarrow{f_{n-2}}\cdots\xrightarrow{f_{2}}\mathsf{M}_1\xrightarrow{f_1}\mathsf{M}_0\rightarrow 0.$$ Let $R$ be a commutative ring with identity. Given a complex $\mathcal{M}=(M_n, f_n)_{n\geq0}$ of $R$-modules, $\R(\mathcal{M})=(\{\R(M_n)\},\{\R(f_n)\})_{n\geq0}$ is a complex of $\R(R)$-semimodules. Let  $\mathcal{M}=(M_n, f_n)_{n\geq0}$ and $\M'=(M'_n, f'_n)_{n\geq0}$ be two complexes of $R$-modules. A homomorphism  $\phi: \mathcal{M}\rightarrow\mathcal{M'}$ of complexes is a collection $\{\phi_n\}_{n\geq 0}$ of $R$-module homomorphisms $\phi_n: {M}_n\rightarrow{M}'_n$ satisfying $ f'_n\phi_n=\phi_{n-1}f_n$ for all $n\geq 1$. Thus if $\phi$ is such a homomorphism, by the functoriality of $\R$, the following diagram is commutative for all $n\geq1$:
\begin{displaymath}
\xymatrix{
\R(M_n) \ar[r]^{\R(\phi_n)} \ar[d]_{\R(f_n)} & \R(M'_n)\ar[d]^{\R(f'_n)}\\
\R(M_{n-1})\ar[r]_{\R(\phi_{n-1})} &\R(M'_{n-1}) 
}
\end{displaymath} 
\vspace{-0.7 cm}
\[{ \operatorname{Diagram}\ 1}\]
 Considering $Z_n(\R(\M)):=\ke \R(f_n)$ as the \textit{$n$-cycles} of $\R(M)$ and $ B_n(\R(\M)):=\im(\R( f_{n+1})$ as the \textit{$n$-boundareis} of $\R(\M)$, we have $B_n(\R(\M))\subseteq Z_n(\R(\M))$ for all $n\geq0$. Now by the \textit{$n$-th radical homology module}, we mean the qoutiont $\R(R)$-semimodule $H_n(\R(\M)):=Z_n(\R(\M))/B_n(\R(\M))$ ( the Bourne classes of $n$-cycles modulo the subsemimodule $B_n(\R(\M))$ of $\R(M_n)$). A complex $\R(\M)$ is said to be \textit{acyclic} if $H_n(\R(\M))=0$ for all $n\geq 0$. It is shown $H_n(\R(-))$ is a covariant functor from the $\boldsymbol{\sf{Ch}}(R{-}\boldsymbol{\sf{Mod}})$ (the category of chain complexes of $R$-modules) to $\R(R){-}\boldsymbol{\sf{Semod}}$ (Theorem \ref{RH}). It is provided conditions under which any short exact sequence $0\rightarrow\R(\M')\xrightarrow{\R(\phi)} \R(\M)\xrightarrow{\R(\psi)} \R(\M'')\rightarrow0$ of complexes of $\R(R)$-semimodules induces a long exact sequence: $$\cdots\rightarrow H_n(\R(\M'))\xrightarrow {H_n(\R(\phi)}H_n( \R(\M)) \xrightarrow {H_n(\R(\psi))}H_n( \R(\M''))\xrightarrow{\alpha_n} H_{n-1}(\R(\M'))\rightarrow\cdots$$ of $\R(R)$-semimodules (Theorem \ref{snake}). In particular, in this case, if any two of $\R(\M')$, $\R(\M)$, and $\R(\M'')$ are acyclic, then the third is also cyclic (Corollary \ref{cyclic}). Let $\mathcal{M}=(\mathsf{M}_n, f_n)_{n\geq0}$ and $\M'=(\mathsf{M}'_n, f'_n)_{n\geq0}$ be two complexes of $\mathsf{R}$-semimodules and  $\mathsf{R}$-semimodule homomorphisms, and $\phi, \psi:\M\rightarrow \M'$ are two maps of complexes. A \textit{homotopy} between $\phi$ and $\psi$ is a pair $(s,t)$, where $s=\{s_n:\mathsf{M}_n\rightarrow\mathsf{M'}_n\}_{n\geq 0}$ and $t= \{t_n:\mathsf{M}_n\rightarrow\mathsf{M'}_n\}_{n\geq 0}$ are two sequences of  $\mathsf{R}$-semimodule homomorphisms, such that $\phi_n+s_{n-1}f_n+f'_{n+1}s_n=\psi_n+t_{n-1}f_{n}+f'_{n+1}t_n$ for all $n\geq 0$. In this case, we write $ \phi\overset{(s, t)}{\simeq}\psi$ to denote homotopy and say $\phi$ and $\psi$ are \textit{homotopic}. Note that, by letting  $\M$ and $\M'$ as chain complexes of $R$-modules and considering $t_n=0$ for all $n\geq 1$ in the last equality, we find that $\phi$ and $\psi$ are homotopic as the maps of complexes of $R$-modules, and write $ \phi\overset{s}{\simeq}\psi$.  In this setting, we get that ${\R(\phi){\overset{(\R(s), \R(0))}{\simeq}}\R(\psi)}$ and  $H_n(\R(\phi))=H_n(\R(\psi))$ (Theorem \ref{Rhomo}). In particular, if $\M$ and $\M'$ are homotopically equivalent as $R$-modules complexes,  then for any  $n\geq0$, the $\R(R)$-semimodules $H_n(\R(\M))$ and $H_n(\R(\M'))$ are isomorphic (Corollary \ref{he}). We say that an $R$-module $P$ is radical projective if $\R(P)$ is a projective $\R(R)$-semimodule, meaning $\operatorname{Hom}(\R(P),-)$ preserves any short exact sequence of $\R(R)$-semimodules. A \textit{radical projective resolution} of an $R$-module $M$ is a complex $\mathcal{P}:=(P_n,f_n)_{n\geq 0}$ of radical projective $R$-modules $P_n$ and $R$-module homomorphisms $f_n$, with an $R$-module homomorphism $g:P_0\rightarrow M$, such that the sequence \[\cdots\rightarrow\R(P_n)\xrightarrow{\R(f_n)}\R(P_{n-1})\xrightarrow{\R(f_{n-1})} \R(P_{n-2})\cdots\rightarrow \R(P_{1})\xrightarrow{\R(f_1)}\R(P_0)\xrightarrow{\R(g)}\R(M)\rightarrow0\] is an exact sequence of $\R(R)$-semimodules. It is shown that any $R$-module homomorphism $g: M\rightarrow M'$ uniquely lifts to a map $\phi: \R(\mathcal{P})\rightarrow \R(\mathcal{P'})$ of complexes of projective $\R(R)$-semimodules, up to homotopy (Theorem \ref{exist}).
 
\section{Homology of semimodules of radical submodules}
Let us begin the section by showing that the $n$-th homology $H_n$ can be regarded as a covariant functor from $\boldsymbol{\sf{Ch}}(R{-}\boldsymbol{\sf{Mod}})$ to $\R(R){-}\boldsymbol{\sf{Semod}}$. 
\begin{thm}\label{RH}
Let  $\M=(M_n,f_n)_{n\geq0}$ and $\M'=(M'_n,f'_n)_{n\geq0}$ be complexes of $R$-modules and $\phi:\M\rightarrow \M':=\lbrace\phi_n:M_n\rightarrow M'_n\rbrace_{n\geq 0}$ be a mapping of complexes of $R$-modules. Then $\phi$ induces a sequence \[\{H_n(\R(\phi)):H_n(\R(\M))\rightarrow H_n(\R(\M'))\}_{n\geq 0}\] of $\R(R)$-semimodule homomorphisms such that for each $N\in \R(M_n)$, \[H_n(\R(\phi))(N/B_n(\R(M))=\R(\phi_n)(N)/B_{n}(\R(M'))=\rad(\phi_n(N))/B_{n}(\R(M)).\] Moreover 
\begin{itemize}
\item[(1)] For the identity map $I_{\M}:\M\rightarrow \M$ of complexes of $R$-modules,  we have $H_n(\R(I_{\M}))=I_{H_n(\R(\M))}$ for all $n\geq 0$, where $I_{H_n(\R(\M))}$ is the identity map on $H_n(\R(\M))$. \\
\item[(2)] For any two maps $\phi:\M\rightarrow\M'$ and $\psi:\M'\rightarrow\M''$ of complexes of $R$-modules, $H_n(\R(\psi))H_n(\R(\phi))=H_n(\R(\psi\phi))$ for all $n\geq 0$.
\end{itemize}
\end{thm}
\begin{proof}
Letting the map $\phi:=\lbrace\phi_n:M_n\rightarrow M'_n\rbrace_{n\geq 0}$ of complexes of $R$-modules from $\M=(M_n,f_n)_{n\geq0}$ to  $\M'=(M'_n,f'_n)_{n\geq0}$ and considering the commutativity of $\operatorname{Diagram} 1$, by the functoriality of $\R$ we have \[\R(f'_n)\R(\phi_n)(Z_n(\R(\M)))=\R(\phi_{n-1})\R(f_n)(Z_n(\R(\M)))=0\] wich shows that $\R(\phi_n)(Z_n(\R(\M)))\subseteq Z_n(\R(\M'))$. 
Hence $\R(\phi_n)$ defines the sequence  $\{H_n(\R(\phi)):H_n(\R(\M))\rightarrow H_n(\R(\M'))\}_{n\geq 0}$ of $\R(R)$-semimodule homomorphisms such that for each $N\in Z_n(\R(\M))$, $H_n(\R(\phi))(N/B_n(\R(M))=\R(\phi_n)(N)/B_{n-1}(\R(M))$. To complete the proof of well-definedness of $\{H_n(\R(\phi))\}$, we assume that $N_1/B_n(\R(\M))=N_2/B_n(\R(\M))$ for $N_1, N_2\in Z_n(\R(\M))$. Then, there exist $L_1,L_2\in B_n(\R(\M))$ such that $ \R(\phi_n)(L_1\+N_1)=\R(\phi_n)(L_2\+N_2)$. Now, since $\R(\phi_n)(L_i)\in B_n(\R(\M'))$, we get that $\R(\phi_n)(N_1)/B_n(\R(\M'))=\R(\phi_n)(N_2)/B_n(\R(\M')) $, as required. Finally, verifying that properties (1) and (2) hold is straightforward.
\end{proof}
Let $\mathcal{M}=(\mathsf{M}_n, f_n)_{n\geq0}$,  $\M'=(\mathsf{M}'_n, f'_n)_{n\geq0}$ and $\M''=(\mathsf{M}''_n, f''_n)_{n\geq0}$ be complexes of $\mathsf{R}$-semimodules. Let $\phi=\{\phi_n: \mathsf{M}'_n\rightarrow \mathsf{M}_n\}$ and $\psi=\{\psi_n: \mathsf{M}_n\rightarrow\mathsf{M}''_n\}$ be two maps of complexes. We say that the sequence \[0\rightarrow\M'\xrightarrow{\phi} \M\xrightarrow{\psi} \M''\rightarrow0\] is an exact sequence of complexes if the sequences \[0\rightarrow \mathsf{M'_n}\xrightarrow{\phi_n}\mathsf{M_n}\xrightarrow{\psi_n} \mathsf{M''_n}\rightarrow0\] are exact sequences of $\mathsf{R}$-semimodules for all $n\geq 0$. Using these notations, we establish the following result:
\begin{thm} \label{snake}
Let $\mathcal{M}=(M_n, f_n)_{n\geq0}$,  $\M'=(M'_n, f'_n)_{n\geq0}$ and $\M''=(M''_n, f''_n)_{n\geq0}$ be complexes of $R$-modules and  $\phi=\{\phi_n: M'_n\rightarrow M_n\}$ and $\psi=\{\psi_n:M_n\rightarrow M''_n\}$ be two maps of complexes. We assume that $0\rightarrow\R(\M')\xrightarrow{\R(\phi)} \R(\M)\xrightarrow{\R(\psi)} \R(\M'')\rightarrow0$ is an exact sequence of complexes of $\R(R)$-semimodules such that for any positive integer $n$,  we have $\im\R(\phi_n)\subseteq \im\R(f_{n+1})$, $\im\R(f''_{n})$ is subtractive and the $\R(R)$-semimodule homomorphism $\R(\psi_n)$ is steady. Then there exists a connecting homomorphism $\alpha_n: H_n(\R(\M''))\rightarrow H_{n-1}(\R(\M'))$ such that the sequence $$\cdots\rightarrow H_n(\R(\M'))\xrightarrow {H_n(\R(\phi))}H_n( \R(\M)) \xrightarrow {H_n(\R(\psi))}H_n( \R(\M''))\xrightarrow{\alpha_n} H_{n-1}(\R(\M'))\rightarrow\cdots$$ is exact.
\end{thm}
\begin{proof}
Since the given sequence of complexes is exact, we have a commutative diagram 
\begin{displaymath}
\xymatrix{
0\ar[r]&\R(M'_{n}) \ar[r]^{\R(\phi_n)} \ar[d]_{\R(f'_n)}& \R(M_n)\ar[d]_{\R(f_{n})} \ar[r]^{\R(\psi_n)}&\R(M''_{n})   \ar[d]_{\R(f''_{n})}\ar[r]&0\\0\ar[r]&
\R(M'_{n-1}) \ar[r]_{\R(\phi_{n-1})} & \R(M_{n-1}) \ar[r]_{\R(\psi_{n-1})}&\R(M''_{n-1})\ar[r]&0
}
\end{displaymath}
with exact rows. This induces a commutative diagram \begin{displaymath}
\xymatrix{
&\R(M'_{n})/B_n(\R(\M')) \ar[r]^{\widehat{\R(\phi_n)}} \ar[d]_{\overline{\R(f'_n)}}& \R(M_n)/B_n({\R(\M)})\ar[d]_{\overline{\R(f_n)}} \ar[r]^{\widehat{\R(\psi_n)}}&\R(M''_{n})/B_n\R(\M''))   \ar[d]_{\overline{\R(f''_n)}}\ar[r]&0\\0\ar[r]&
Z_{n-1}(\R(\M')) \ar[r]_{\overline{\R(\phi_{n-1})}} & Z_{n-1}(\R(\M)) \ar[r]_{\overline{\R(\psi_{n-1})}}&Z_{n-1}(\R(\M''))
}
\end{displaymath} 
where $\widehat{\R(\phi_n)}:\R(M'_n)/B_n(\R(\M'))\rightarrow\R(M_n)/B_n(\R(\M))$ is defined by  \[\widehat{\R(\phi_n)}(N'_n/B_n(\R(\M')))=\R(\phi_n)(N'_n)/B_n(\R(\M)).\] Since $ \R(\phi_n)\R(f'_{n+1})=\R(f_{n+1})\R(\phi_{n+1})$, the map $\widehat{\R(\phi_n)}$ is well-defined. Similarly, $\widehat{\R(\psi_n)}$ is defined in the same manner. We proceed to show that the rows of diagram are exact. For the first row, we observe that \[(\widehat{\R(\psi_n)}\widehat{\R(\phi_n)})(N'_n/B_n(\R(\M')))=(\R(\psi_n)((\R(\phi_n)(N'_n))/B_n(\R(\M''))=0\]
for every $N'_n\in\R(M'_n)$. Thus, $\im\widehat{\R(\phi_n)}\subseteq \ke\widehat{\R(\psi_n)}$. For the reverse inclusion, let $N_n/B_n(\R(\M))\in \ke\widehat{\R(\psi_n)}$. Then $\R(\psi_n)(N_n)/B_n(\R(\M''))=\widehat{\R(\psi_n)}(N_n/B_n(\R(\M)))=0$, and so there exist $K''_1, K''_2\in B_n(\R(\M''))$ such that $\R(\psi_n)(N_n)\+K''_1=K''_2$. Since $\im\R(f''_{n+1})$ is subtractive, we have $\R(\psi_n)(N_n)\in B_n(\R(\M''))$. Therefore, there exists $N''_{n+1}\in \R(M''_{n+1})$ such that $\R(\psi_n)(N_n)= \R(f''_{n+1})(N''_{n+1})$, and since $\R(\psi_{n+1})$ is surjective, there exists $N_{n+1} \in \R(M_{n+1})$  such that $\R(\psi_{n+1})(N_{n+1})=N''_{n+1}$. Thus   $\R(\psi_n)(N_n)=\R(f''_{n+1})\R(\psi_{n+1})(N_{n+1}) =\R(\psi_n)\R(f_{n+1})(N_{n+1})$. Since $\R(\psi_n)$ is steady, there are $K_1, K_2\in \ke\R(\psi_n)=\im\R(\phi_n)$ so that $N_n\+K_1=\R(f_{n+1})(N_{n+1})\+K_2$. In particular, we have  \[N_n\+K_1\+\rad (f_{n+1})(M_{n+1})=\R(f_{n+1})(N_{n+1})\+K_2\+ \rad(f_{n+1})(M_{n+1}).\] Now,  since $\im\R(\phi_n)\subseteq \im\R(f_{n+1})=B_n(\R(\M))$,  there exists $N'_n\in\R(M'_n)$ such that $\R(\phi_n)(N'_n)=K_2$, and so we get \[N_n/B_n(\R(\M)=\R(\phi_n)(N'_n)/B_n(\R(\M))=\widehat{\R(\phi_n)}(N'_n/B_n(\R(\M'))),\] which implies $N_n/B_n(\R(\M))\in \im\widehat{\R(\phi_n)}$. Thus, $\ke\widehat{\R(\psi_n)}=\im\widehat{\R(\phi_n)}$, establishing the exactness of the first row. To show the exactness of the second row, we define $\overline{\R(\phi_{n-1})}: Z_{n-1}(\R(\M'))\rightarrow Z_{n-1}(\R(\M))$ by $\overline{\R(\phi_{n-1})}(N'_{n-1})=\R(\phi_{n-1})(N'_{n-1})$ for every $N'_{n-1}\in\R(M'_{n-1})$. Now, since $\R(\phi_{n-2})\R(f'_{n-1})=\R(f_{n-1})\R(\phi_{n-1})$, $\overline{\R(\phi_{n-1})}$ is well-defined. Similarly, define $\overline{\R(\psi_{n-1})}: Z_{n-1}(\R(\M))\rightarrow Z_{n-1}(\R(\M''))$. Clearly $\im\overline{\R(\phi_{n-1})}\subseteq \ke\overline{\R(\psi_{n-1})}$. On the other hand, if $N_{n-1}\in \ke\overline{\R(\psi_{n-1})}$, then $\R(\psi_{n-1})(N_{n-1})=0$ and so $N_{n-1}\in \ke\R(\psi_{n-1})=\im\R(\phi_{n-1})$. Therefore there exists $N'_{n-1}\in \R(M'_{n-1})$ such that $N_{n-1}=\R(\phi_{n-1})(N'_{n-1})$ and $ N_{n-1}\in Z_{n-1}(\R(\M))$. Since \[\R(\phi_{n-2})\R(f'_{n-1})(N'_{n-1})=\R(f_{n-1})\R(\phi_{n-1})(N'_{n-1})=\R(f_{n-1})(N_{n-1})=0\] and $\ke\R(\phi_{n-1})={0}$ we get that $N'_{n-1}\in Z_{n-1}(\R(\M'))$, and so  $N_{n-1}=\R(\phi_{n-1})(N'_{n-1})=\overline{\R(\phi_{n-1})}(N'_{n-1})\in \im\overline{\R(\phi_{n-1})}$. In the diagram, the vertical maps are induced by the boundary operators, i.e.,  $\overline{\R(f_n)}:\R(M_n)/B_n(\R(\M))\rightarrow Z_{n-1}(\R(\M))$ is given by \[\overline{\R(f_n)}(N_n/B_n(\R(\M)))=\R(f_n)(N_n)\in \im\R(f_n)\subseteq \ke\R(f_{n-1})=Z_{n-1}(\R(\M)).\] So $\overline{\R(f_n)}$ is well-defined. Clearly $N_n/B_n(\R(\M))\in \ke\overline{\R(f_n)}$ if and only if $\R(f_n)(N_n)=0$  if and only if $ N_n\in Z_n(\R(\M))$  so that $\ke\overline{\R(f_n)}=H_n(\R(\M))$. Also $\operatorname{coker}\overline{\R(f_n)}=Z_n(\R(\M))/ \im\overline{\R(f_n)}=H_{n-1}(\R(\M))$. Hence, by \cite [Theorem 3.10]{jawad} there exists a connecting semimodule homomorphism $\alpha_n : H_n(\R(\M''))\rightarrow H_{n-1}(\R(\M'))$ such that the sequence $$\cdots\rightarrow H_n(\R(\M'))\xrightarrow{H_n(\R(\phi))}H_n( \R(\M))\xrightarrow{H_n(\R(\psi))}H_n( \R(\M''))\xrightarrow{\alpha_n} H_{n-1}(\R(\M'))\rightarrow\cdots$$ is exact. 
\end {proof}
From the above Theorem, the following conclusion can be drawn.
\begin{cor}\label{cyclic}
Let $0\rightarrow\R(\M')\xrightarrow{\R(\phi)} \R(\M)\xrightarrow{\R(\psi)} \R(\M'')\rightarrow0$ be an exact sequence of complexes of $\R(R)$-semimodules such that for any positive integer $n$, $\im\R(f''_{n})$ is subtractive, $\R(\psi_n)$ is steady, and $\im\R(\phi_n)\subseteq \im\R(f_{n+1})$. If any two of three complexes $\R(\M')$, $\R(\M)$ and $\R(\M'')$ are acyclic, then the third is also acyclic. 
\end{cor}
The following result illustrates the naturality condition for exact homology sequences.
\begin{cor}
Let the diagram \begin{displaymath}
\xymatrix{
	0\ar[r]&\R(\M') \ar[r]^{\R(\phi)} \ar[d]_{\R(\beta')}& \R(\M)\ar[d]_{\R(\beta)} \ar[r]^{\R(\psi)}&\R(\M'')   \ar[d]_{\R(\beta'')}\ar[r]&0\\0\ar[r]&
	\R(\mathcal{N'}) \ar[r]_{\R(\phi')} & \R(\mathcal{N}) \ar[r]_{\R(\psi')}&\R(\mathcal{N''})\ar[r]&0
}
\end{displaymath} be  commutative where the rows are exact sequences of complexes having similar properties as in Theorem \ref{snake} and the vertical maps are mappings of complexes, the induced diagram
 \begin{displaymath}
\xymatrix{
\cdots\ar[r]&H_n(\R(\M')) \ar[r]^{H_n(\R(\phi))} \ar[d]_{H_n(\R(\beta'))}& H_n(\R(\M))\ar[d]_{H_n(\R(\beta))} \ar[r]^{H_n(\R(\psi))}&H_n(\R(\M''))   \ar[d]_{H_n(\R(\beta''))}\ar[r]^{\alpha_n}&H_{n-1}(\R(\M')) \ar[d]_{H_{n-1}(\R(\beta'))}\ar[r]&\cdots\\\cdots\ar[r]&
H_n(\R(\mathcal{N'})) \ar[r]_{H_n(\R(\phi'))} &H_n(\R(\mathcal{N})) \ar[r]_{H_n(\R(\psi'))}&H_n(\R(\mathcal{N''}))\ar[r]_{\alpha'_n}&H_{n-1}(\R(\mathcal{N'}))\ar[r]&\cdots
}
\end{displaymath}
 is commutative.	
\end{cor}
\begin{proof}
Straightforward.
\end{proof}
\section{A homotopy pair between maps of complexes}
Let $\mathcal{M}=(\mathsf{M}_n, f_n)_{n\geq0}$ and $\M'=(\mathsf{M}'_n, f'_n)_{n\geq0}$ be two complexes of $\mathsf{R}$-semimodules and  $\mathsf{R}$-semimodule homomorphisms. Any two maps of complexes $\phi$ and $\psi$ from $\M$ to $\M'$ are saied to be \textit{homotopic}  if there exist two sequences $s=\{s_n:\mathsf{M}_n\rightarrow\mathsf{M'}_n\}_{n\geq 0}$ and $t= \{t_n:\mathsf{M}_n\rightarrow\mathsf{M'}_n\}_{n\geq 0}$ of $\mathsf{R}$-semimodule homomorphisms  such that \[\phi_n+s_{n-1}f_n+f'_{n+1}s_n=\psi_n+t_{n-1}f_{n}+f'_{n+1}t_n\ \ (\ast) \] for all $n\geq 0$. In this case, $ (s,t) $  is called a \textit{homotopy pair} (or a \textit{homotopy} for short)  from  $\phi$ to $\psi$  and we write $ \phi\overset{(s, t)}{\simeq}\psi$. Note that, by letting  $\M$ and $\M'$ as chain complexes of $R$-modules and considering $t_n=0$ for all $n\geq 1$ in $(\ast)$, the equation reduces to the standard definition of homotopy between maps of complexes of $R$-modules, where we write $ \phi\overset{s}{\simeq}\psi$. 
\begin{thm}\label{Rhomo}
Let $\phi$ and $\psi$ be two maps of complexes of $R$-modules such that $\phi\overset{s}{\simeq}\psi$. Then $\R(\phi)\overset{(\R(s), \R(0))}{\simeq}\R(\psi)$ and  $H_n(\R(\phi))=H_n(\R(\psi))$.
\end{thm}
\begin{proof}
Considering the homotpicity $\phi\overset{s}{\simeq}\psi$ between the maps of complexes of $R$-modules $\phi, \psi: \M\rightarrow\M'$, and applying the functor $\R$ to both sides of the equality $(\ast)$, we obtain $\R(\phi)\overset{(\R(s), \R(0))}{\simeq}\R(\psi)$. In particular, if $N\in Z_n(\R(M))$, then $\R(f_n(N))=0$, and hence
\[\R( \phi_n)(N)\+(\R(f'_{n+1})\R(s_n))(N)=\R(\psi_n)(N),\]
 where $\R(f'_{n+1}s_n)(N)\in B_n(\R(M'))$. Thus we conclude that  $H_n(\R(\phi))(N/B_n(\R(\M)))=H_n(\R(\psi))(N/B_n(\R(\M))$. 
 \end{proof}  
 A complex $\M$ of $R$-modules is said to be \textit{radical acyclic} if $H_n(\R(\M))=0$ for all $n\geq 0$.
\begin{cor}
Let  $\M$ be a complex of $R$-modules. If  the identity map and zero map on $\M$ are homotopic, then the complex $\M$ is radical acyclic.
\end{cor}
\begin{proof}
Suppose that  the identity map $I$ and the zero map $0$ be homotopic on $\M$. By Theorem \ref{Rhomo}, it follows that $H_n(\R(I))=H_n(\R(0))$, and so $I_{H_n(\R(\M))}=0_{H_n(\R(\M)}$ for all $n\geq 0$. Thus we conclude that $H_n(\R(\M))=0$  for all $n\geq 0$.

\end{proof}
Any two complexes $\M$ and  $\M'$ of $R$-modules are saied to be \textit{homotopically equivalent} if there exist two maps $\phi:\M\rightarrow\M'$ and $\psi: \M'\rightarrow\M$ such that $\phi\psi\simeq I_{\M'}$ and $ \psi\phi\simeq I_{\M}.$
\begin{cor}\label{he}
Let $\M$ and $\M'$ be two homotopically equivalent complexes of $R$-modules. Then for any  $n\geq0$, the $\R(R)$-semimodules $H_n(\R(\M))$ and $H_n(\R(\M'))$ are  isomorphic.
\end{cor}
\begin{proof}
Assume that $\phi:\M\rightarrow\M'$ is a homotopy equivalence of $R$-modules. Then there exists a map $\psi: \M'\rightarrow\M$ such that $\phi\psi\simeq I_{\M'}$ and $ \psi\phi\simeq I_{\M}$. So by Theorem \ref{RH}(2), $\R(\phi)\R(\psi)\simeq I_{\R(\M')}$ and $\R(\psi)\R(\phi)\simeq I_{\R(\M)}$. Therefore by Theorem \ref{Rhomo},  $H_n(\R(\phi)): H_n(\R(\M))\rightarrow H_n(\R(\M'))$ is an $\R(R)$-semimodule isomorphism for all $n\geq0$.
\end{proof}
We say that an $R$-module  $P$ is \textit{ radical projective} if $\R(P)$ is a projective  $\R(R)$-semimodule, i.e., for every short exact sequence of  $\R(R)$-semimodules \[0\rightarrow\mathsf{M'}\xrightarrow{f} \mathsf{M}\xrightarrow{g} \mathsf{M''}\rightarrow 0,\] the induced sequence \[0\rightarrow \operatorname{Hom}(\R(P),\mathsf{M'})\xrightarrow{f^*}\operatorname{Hom}(\R(P),\mathsf{M})\xrightarrow{g^*} \operatorname{Hom}(\R(P),\mathsf{M''})\rightarrow0\] is an exact sequence of $\R(R)$-semimodules. It is worth noting that our definition of the projectivity of a semimodule is consistent with that presented in \cite[]{Golan}. However, alternative approaches to semimodule projectivity have been explored in \cite{flatjawad,patch}.  \\
Let $M$ be an $R$-module. A \textit{radical projective resolution} of $M$ is a complex $\mathcal{P}:=(P_n,f_n)_{(n\geq 0)}$ of $R$-modules $P_n$ and $R$-module homomorphisms $f_n$, together with an $R$-module homomorphism $g:P_0\rightarrow M$ such that the sequence \[\cdots\rightarrow\R(P_n)\xrightarrow{\R(f_n)}\R(P_{n-1})\xrightarrow{\R(f_{n-1})} \R(P_{n-2})\cdots\rightarrow \R(P_{1})\xrightarrow{\R(f_1)}\R(P_0)\xrightarrow{\R(g)}\R(M)\rightarrow0\] is an exact sequence of $\R(R)$-semimodules and all the $P_i$ are radical projective.  Clearly, the complex $\mathcal{P}$, satisfies $H_n(\R(\mathcal{P}))=0$ for all $n\geq1$ and \[H_0(\R(\mathcal{P}))=Z_0(\R(\mathcal{P}))/B_0(\R(\mathcal{P}))=\ke\R(f_0)/\im\R(f_1)=\R(P_0)/\ke g\simeq\R(P).\] We sometimes use $\R(\mathcal{P})\rightarrow\R(M)\rightarrow 0$ to denote $\mathcal{P}$ is a radical projective resolution of $M$.
\begin{example} Let the notations be as before.
 \begin{itemize}
\item[(1)]
Considering  $P$ a radical projective $R$-module, we have always a trivial radical projective resolution by setting $P_0=P$, $P_n=0$ and $f_n=0$ for $n\geq1$. Here, $g$ is assumed to be the identity homomorphism on $P$.
\item[(2)]
Let $F$ be a field, $V$ an $F$-vector space, $ W $ a subspace of $V$ and $P=V/W$. Then by \cite[Corollary 24]{smr} the sequence $0\rightarrow\R(W)\xrightarrow{\R(i)} \R(V)\xrightarrow{\R(\pi)}\R(P)\rightarrow 0$ is exact, where $i:W\rightarrow V$ is the inclusion map and $\pi:V\rightarrow P$ is the natural projection map. So by considering $P_0=P$, $P_1=V$ and $P_2=W$ and $P_n=0$ for all $n\geq 3$, and setting $g=\pi$, $f_1=i$ and  $f_n=0$ for all $n\geq 2$, we see that  $\mathcal{P}=(P_n,f_n)$ forms a projective resolution of $\R(P)$.
\end{itemize}
\end{example}
As stated in \cite[Example 17.1]{Golan}, for any semiring $\mathsf{R}$ and any nonempty set $A$, the subset $\{\chi_a\}_{a\in A}$ considering of characterristic functions on any element $a\in A$, forms a basis for the free $\mathsf{R}$-semimodule $\mathsf{R}^{(A)}$. Thus by \cite[Proposition 17.12]{Golan}, we see that if $\mathsf{M}$ is an $\mathsf{R}$-semimodule and $A$ is a nonempty set, and if $g$ is a function from $A$ to $\mathsf{M}$, then there exists a unique $\mathsf{R}$-semimodule homomorphism $\alpha: \mathsf{R}^{(A)}\rightarrow \mathsf{M}$ given by $\chi_a\mapsto g(a)$ for all $a\in A$. This fact is used in the following result.
\begin{lemma}\label{pro}
Let  $M$ and $M'$ be $R$-modules. Then, if $f:M\rightarrow M'$ is an $R$-module homomorphism such that $\R(f)$ is steady and if for any radical projective $R$-module $P$, the mappings $\alpha, \alpha': \R(P)\rightarrow\R(M)$ are $\R(R)$-semimodule homomorphisms such that $\R(f)\alpha=\R(f)\alpha'$, then there are $\R(R)$-semimodule  homomorphisms $\beta, \beta':\R(P)\rightarrow\R(M)$ such that $ \alpha+\beta=\alpha'+\beta'$ and   $ \R(f)\beta=\R(f)\beta'=0$ 
\end{lemma}
\begin{proof}
Since $P$ is a radical projective $R$-module, by \cite[Proposition 17.16]{Golan}, $\R(P)$ is a retract of a free $\R(R)$-semimodule, i.e., there are a surjective $\R(R)$-semimodule homomorphism $\phi:(\R(R))^{(\Gamma)}\rightarrow\R(P)$ and an injective $\R(R)$-semimodule homomorphism $\psi:\R(P)\rightarrow(\R(R))^{(\Gamma)}$ for some indexe set $\Gamma$  such that $\phi\psi=I_{\R(P)}$. Now, since  $\R(f)\alpha\phi=\R(f)\alpha'\phi$ and $\R(f)$ is steady, we get that for every  $\gamma\in \Gamma$  there are $N_{\gamma}, N'_{\gamma}\in\ke(\R(f))$ such that $\alpha\phi(\gamma)\+N_{\gamma}=\alpha'\phi(\gamma)\+N'_{\gamma}$. By \cite[Proposition 17.12]{Golan}, there are  unique $\R(R)$-semimodule homomorphisms $\lambda, \lambda':\R(R))^{(\Gamma)}\rightarrow \R(M)$ such that $\lambda(\gamma)=N_{\gamma}$ and  $\lambda'(\gamma)=N'_{\gamma}$ for every $\gamma\in\Gamma$. Thus by considering $\beta:=\lambda\psi$ and $\beta':=\lambda'\psi$, we have $$\R(f)\beta=\R(f)(\lambda\psi)=(\R(f)\lambda)\psi=0=(\R(f)\lambda')\psi=\R(f)(\lambda'\psi)=\R(f)\beta'.$$ Moreover, for each $\gamma\in\Gamma$ we have $$(\alpha\phi+\lambda)(\gamma)=(\alpha\phi)(\gamma)+N_{\gamma}=\alpha'\phi(\gamma)\+N'_{\gamma}=(\alpha'\phi+\lambda')(\gamma),$$ whence $\alpha\phi+\lambda=\alpha'\phi+\lambda'$. Hence we conclude that 
\begin{align*}
 \alpha+\beta &=\alpha+\lambda\psi =\alpha  I_{\R(P)}+\lambda\psi = \alpha(\phi\psi)+\lambda\psi\\ &=(\alpha\phi+\lambda)\psi=(\alpha'\phi+\lambda')\psi =\alpha'(\phi\psi)+\lambda'\psi\\ &=\alpha' I_{\R(P)}+\lambda'\psi=\alpha'+\lambda'\psi=\alpha'+\beta'.
\end{align*}
\end{proof}
Next, we show that for any $R$-module homomorphism $g:M\rightarrow M'$, $\R(g)$ can be lifted to a mapping $\phi: \R(\mathcal{P})\rightarrow \R(\mathcal{P'})$ of complexes of $\R(R)$-semimodules and this lifting map is unique up to homotopy, where $\mathcal{P}=(P_n,f_n)_{n\geq 0}$ and $\mathcal{P'}=(P'_n,f'_n)_{n\geq 0}$ are radical projective resolutions of $M$ and $M'$ respectively.  
\begin{thm} \label{exist}
Let $\mathcal{P}=(P_n,f_n)_{n\geq 0}$ and $\mathcal{P'}=(P'_n,f'_n)_{n\geq 0}$  be radical projective resolutions of $M$ and $M'$, respectively. Then for any $R$-module homomorphism $g: M\rightarrow M'$, there exists a mapping of complexes $\phi:\R(\mathcal{P})\rightarrow \R(\mathcal{P}')=\{\phi_n:\R(P_n)\rightarrow \R(P'_n)\}_{n\geq 0}$ such that $\R(g)\R(f)=\R(f')\phi_0$. Moreover, $\phi$ is unique up to a homotopy.
\end{thm}
\begin{proof}
We first prove the existence of $\phi$. According to the diagram
 \begin{displaymath}
\xymatrix{
\R(P_0) \ar[r]^{\R(f)}  \ar[ddr]_{\R(g)\R(f)} & \R(M)\ar[dd]^{\R(g)}\ar[r]&0\\&\\
\R(P'_{0})\ar[r]_{\R(f')} &\R(M')\ar[r]&0
}
\end{displaymath}
and radical projectivity of $P_0$ and surjectivity of $\R(f')$, there exsists an $\R(R)$-semimodule homomorphism $\phi_0:\R(P_0)\rightarrow\R(P'_0)$ such that  $\R(f')\phi_0=\R(g)\R(f)$. Assume inductively that $\phi_i: \R(P_i)\rightarrow\R(P'_i)$ has been defined for all $0\leq i\leq n$ such that $ \phi_{i-1}\R(f_i)=\R(f'_i)\phi_i$ in which $\phi_{-1}=\R(g)$, $f_0=f$ and $f'_0=f'$. Consider the diagram
 \begin{displaymath}
\xymatrix{
\R(P_{n+1}) \ar[r]^{\R(f_{n+1})} \ar[ddr]_{\phi_{n}\R(f_{n+1})}& \R(P_n)\ar[dd]^{\phi_{n}} \ar[r]^{\R(f_{n})}&\R(P_{n-1})   \ar[dd]_{\phi_{n-1}} \\&\\
\R(P'_{n+1}) \ar[r]_{\R(f'_{n+1})} & \R(P'_n) \ar[r]_{\R(f'_{n}) }&\R(P'_{n-1})
}
\end{displaymath}
We see that $\im(\phi_n\R(f_{n+1}))\subseteq\im\R(f'_{n+1})=\ke\R(f'_{n})$ because
\begin{center}
 $\R(f'_n)\phi_n\R(f_{n+1})=\phi_{n-1}\R(f_n)\R(f_{n+1})=0.$
\end{center}
 Also, by considering the diagram
 \begin{displaymath}
\xymatrix{
&\R(P_{n+1}) \ar[d]^{\phi_{n}\R(f_{n+1})} \\
 \R(P'_n) \ar[r]^{\R(f'_{n})\ \ \ \ \ \ \ }&\R(f'_{n+1})(P'_{n+1})\ar[r]&0
}
\end{displaymath}
the radical projectivity of $P_{n+1}$ gives an $\R(R)$-semimodule homomorphism $\phi_{n+1}:\R(P_{n+1})\rightarrow\R(P'_{n+1})$ such that $\R(f'_{n+1})\phi_{n+1}=\phi_n\R(f_{n+1})$ and this completes the proof by induction.  To show that $\phi$ is unique up to homotopy, suppose there is another mapping of complexes $\psi=\{\psi_n\}_{n\geq 0}$ from $\R(\mathcal{P})$ to $\R(\mathcal{P}')$ such that $\R(g)\R(f)=\R(f')\psi_0$. Assuming $t_{-2}=s_{-2}=t_{-1}=s_{-1}=0$, we have \[\R(f')s_{-1}+s_{-2}\R(f_{-1})+\R(g)= \R(f')t_{-1}+t_{-2}\R(f_{-1})+\R(g).\] Suppose  by induction there are $\R(R)$-semimodule homomorphisms $s_i, t_i: \R(P_i)\rightarrow\R(P'_{i+1})$ for $0\leq i\leq n$, such that $$\R(f'_{i+1})s_{i}+s_{i-1}\R(f_{i})+\phi_i= \R(f'_{i+1})t_{i}+t_{i-1}\R(f_{i})+\psi_i  .$$ In particular, by composing $\R(f_{n+1})$  with both sides of this equality when $i=n$, we have
\begin{align*}
\R(f'_{n+1})s_{n}(\R(f_{n+1}))+s_{n-1}\R(f_{n})(\R(f_{n+1}))+\phi_n(\R(f_{n+1}))=\\ \R(f'_{n+1})t_{n}(\R(f_{n+1}))+t_{n-1}\R(f_{n})(\R(f_{n+1}))+\psi_n(\R(f_{n+1})),
\end{align*}
and so we obtain that \[\R(f'_{n+1})(s_{n}\R(f_{n+1}))+(\R(f'_{n+1}))\phi_{n+1}= \R(f'_{n+1})(t_{n}\R(f_{n+1}))+(\R(f'_{n+1}))\psi_{n+1}.\] Therefore $\R(f'_{n+1})(s_{n}\R(f_{n+1})+\phi_{n+1})=\R(f'_{n+1})(t_{n}\R(f_{n+1})+\psi_{n+1}).$
Now, since $\R(f'_{n+1})$ is steady and $P_{n+1}$ is radical projective, by Lemma \ref{pro}, there are $\R(R)$-semimodule homomorphisms $\alpha_{n+1}, \alpha'_{n+1}: \R(P_{n+1})\rightarrow\R(P'_{n+1})$ such that \[s_{n}\R(f_{n+1})+\phi_{n+1}+\alpha_{n+1}=t_{n}\R(f_{n+1})+\psi_{n+1}+\alpha'_{n+1}\ \ (\ast)\] and $\R(f'_{n+1})\alpha_{n+1}=\R(f'_{n+1})\alpha'_{n+1}=0.$ Hence we have  $\alpha_{n+1}(\R(P_{n+1}))\subseteq \ke(\R(f'_{n+1}))=\R(f'_{n+2})(\R(P'_{n+2}))$ and  
 $\alpha'_{n+1}(\R(P_{n+1}))\subseteq \ke(\R(f'_{n+1}))=\R(f'_{n+2})(\R(P'_{n+2})).$
  Consequently, since $P_{n+1}$ is radical projective and 
   $\R(f'_{n+2}):\R(P'_{n+2})\rightarrow\R(f'_{n+2})(\R(P'_{n+2}))$
   is surjective, there are $\R(R)$-homomorphisms $s_{n+1}, t_{n+1}: \R(P_{n+1})\rightarrow\R(P'_{n+2})$ such that $\alpha_{n+1}=\R(f'_{n+2})s_{n+1} $ and $\alpha'_{n+1}=\R(f'_{n+2})t_{n+1}$. Now, by replacing  $\alpha_{n+1}$ and $ \alpha'_{n+1}$ in $(\ast)$, we conclude that  $\R(f'_{n+2})s_{n+1}+s_{n}\R(f_{n+1})+\phi_{n+1}= \R(f'_{n+2})t_{n+1}+t_{n}\R(f_{n+1})+\psi_{n+1}$, which shows $\phi\overset{(s, t)}{\simeq}\psi$.
\end{proof}
We end with the following immediate result of Theorem \ref{exist}.
\begin{cor}
Every $R$-module isomomorphism uniquely lifts up to homotopy to an isomorphism of complexes of projective $\R(R)$-semimodules.
\end{cor}
\begin{proof}
Let $\mathcal{P}=(P_n,f_n)_{n\geq 0}$ and $\mathcal{P'}=(P'_n,f'_n)_{n\geq 0}$  be radical projective resolutions of $M$ and $M'$, respectively and  $g: M\rightarrow M'$ is an $R$-module isomomormphis such that the map $\phi:\R(\mathcal{P})\rightarrow\R(\mathcal{P'})$ is lifting of the map $\R(g): \R(M)\rightarrow \R(M')$ to the coresponding radical projective resoultion and $\R(f_i)$ and $\R(f'_i)$ are steady. By Theorem\ref{exist} the diagram\begin{displaymath}
 \xymatrix{
	\cdots\ar[r]&\R(P_n) \ar[r]^{\R(f_n)} \ar[d]_{\phi_n}& \R(P_{n-1})\ar[d]_{\phi_{n-1}} \ar[r]^{\R(f_{n-1})}&\cdots\ar[r]^{\R(f_1)} &\R(P_0)   \ar[d]_{\phi_0}\ar[r]^{\R(f_0)}&\R(M) \ar[d]_{\R(g)}\\\cdots\ar[r]&
	\R(P'_n) \ar[r]_{\R(f'_n)} & \R(P'_{n-1}) \ar[r]_{\R(f'_{n-1})}&\cdots \ar[r]_{\R(f'_1)}&\R(p'_0)\ar[r]_{\R(f'_0)}&\R(M')
}
\end{displaymath}  is commutative. Hence by Theorem \ref{exist} there exists the map $\psi:\R(\mathcal{P'})\rightarrow\R(\mathcal{P})$ such that the diagram \begin{displaymath}
 \xymatrix{
	\cdots\ar[r]&\R(P'_n) \ar[r]^{\R(f'_n)} \ar[d]_{\psi_n}& \R(P'_{n-1})\ar[d]_{\psi_{n-1}} \ar[r]^{\R(f'_{n-1})}&\cdots\ar[r]^{\R(f'_1)} &\R(P'_0)   \ar[d]_{\psi_0}\ar[r]^{\R(f'_0)}&\R(M') \ar[d]_{\R(g^{-1})}\\\cdots\ar[r]&
	\R(P_n) \ar[r]_{\R(f_n)} & \R(P_{n-1}) \ar[r]_{\R(f_{n-1})}&\cdots \ar[r]_{\R(f_1)}&\R(p_0)\ar[r]_{\R(f_0)}&\R(M)
}
\end{displaymath} is commutative. Therefore by  \cite[Theorem 4]{smr} and  Theorem\ref{exist} $\phi \psi\simeq I_{\R(P')}$ and  $\psi \phi \simeq I_{\R(P)}$.
\end{proof}

\end{document}